\documentclass[12pt]{amsart}

\usepackage{amsmath, amsthm, amssymb, amsfonts, amscd, graphicx, endnotes, tikz, color, hyperref}

\usetikzlibrary{tikzmark,calc}

\usepackage{mathtools}

\usepackage[margin=01.2in]{geometry}

\usetikzlibrary{arrows}

\usepackage[nodisplayskipstretch]{setspace}
\def\CC{{\mathbb C}}

\def\QQ{{\mathbb Q}}

\def\QQ{{\mathbb Q}}
\def\RR{{\mathbb R}}

\def\ZZ{{\mathbb Z}}

\def\0{{\mathbf 0}}
\def\1{{\mathbf 1}}

\def\Mcal{{\mathcal M}}

\def\Ocal{{\mathcal O}}

\def\Qbar{\overline{\QQ}}
\def\PCF{\mathrm{PCF}}

\def\disc{\mathrm{disc}}

\def\Gal{\mathrm{Gal}}

\def\Res{\mathrm{Res}}

\def\tr{\mathrm{tr}}

\theoremstyle{plain}

\newtheorem{thm}{Theorem}

\newtheorem{cor}[thm]{Corollary}

\newtheorem{lem}[thm]{Lemma}

\theoremstyle{definition}

\title[Totally real algebraic integers and unicritical families]{Totally real algebraic integers in short intervals, Jacobi polynomials, and unicritical families in arithmetic dynamics}

\author{Chatchai Noytaptim}

\address{Chatchai Noytaptim; Department of Mathematics; Oregon State University; Corvallis OR 97331 U.S.A.}

\email{noytaptc@oregonstate.edu}

\author{Clayton Petsche}

\address{Clayton Petsche; Department of Mathematics; Oregon State University; Corvallis OR 97331 U.S.A.}

\email{petschec@math.oregonstate.edu}

\date{November 12, 2022}
\begin{document}

\begin{abstract}
We classify all post-critically finite unicritical polynomials defined over the maximal totally real algebraic extension of $\QQ$.  Two auxiliary results used in the proof of this result may be of some independent interest.  The first is a recursion formula for the $n$-diameter of an interval, which uses properties of Jacobi polynomials.  The second is a numerical criterion which allows one to the give a bound on the degree of any algebraic integer having all of its complex embeddings in a real interval of length  less than $4$.  
\end{abstract}

\maketitle

%%%%%%%%%%%%%%%%%
%%%%%%%%%%%%%%%%%
%%%%%%%%%%%%%%%%%
%%%%%%%%%%%%%%%%%
%%%%%%%%%%%%%%%%%
%%%%%%%%%%%%%%%%%

\section{Introduction}

Let $d\geq2$ be an integer, let $c\in\Qbar$, and consider the polynomial map $f:\Qbar\to\Qbar$ defined by $f(x)=x^d+c$.  Define $f^n:\Qbar\to\Qbar$ to be the $n$-fold composition 
\[
f^n=f\circ f \circ\dots\circ f
\] 
of $f$ with itself.  The polynomial $f(x)$ is said to be post-critically finite, or PCF, if its unique critical point $0$ is preperiodic with respect to iteration; in other words, if the critical orbit $0,f(0),f^2(0),\dots$ takes only finitely many distinct values. It has been well established in complex and arithmetic dynamics that PCF maps are a class well worthy of study, as they often possess dynamical properties which are highly distinguishable from arbitrary maps.  For convenience, we define
\[
\PCF_d=\{c\in\Qbar\mid f(x)=x^d+c\text{ is PCF}\}
\]
to be the set of all PCF parameters $c$ with respect to the degree $d$ unicritical family $x^d+c$.

Given an algebraic number $c\in\Qbar$ of degree $n=[\QQ(c):\QQ]$, let $F_c(X)\in\QQ[X]$ be its minimal polynomial over $\QQ$.  Recall that $c$ is said to be totally real if $F_c(X)$ has $n$ real roots.  For each $d\geq2$, one might ask how often a unicritical polynomial $x^d+c$ can be both PCF and defined over $\QQ^\tr$.  In other words, what can we say about the intersection $\PCF_d\cap\QQ^\tr$, and in particular, is this intersection finite?  Although $\PCF_d$ is a set of bounded height, the field $\QQ^\tr$ is an infinite degree extension of $\QQ$, and therefore such a finiteness result does not follow directly from the Northcott property of heights.

In fact, the finiteness result
\begin{equation}\label{FinitenessIntro}
|\PCF_d\cap\QQ^\tr|<+\infty
\end{equation}
for each $d\geq2$ is not difficult to obtain using a theorem of Fekete in arithmetic capacity theory.  First one observes that the generalized Mandelbrot set 
\[
\Mcal_d=\{c\in\CC\mid \text{the forward orbit of $0$ under $x^d+c$ is bounded}\}
\]
meets the real line at an interval of length $<4$; in the case $d=2$ this follows from the well-known fact that the ordinary Mandelbrot set $\Mcal_2$ meets the real line at the interval $[-2,1/4]$, and in the case $d\geq3$ the comparable calculation has been done by Paris\'e-Rochon \cite{MR3654403} and Paris\'e-Ransford-Rochon \cite{PariseRansfordRochon}.  The second relevant observation is that the set $\PCF_d$ consists entirely of algebraic integers.  Consequently, $\PCF_d\cap\QQ^\tr$ is confined to an adelic set of arithmetic capacity $<1$, and hence must be finite.  (A standard source for arithmetic capacity theory is Rumely \cite{MR1009368}.)  We give more of the details in the proof of Theorem~\ref{FinitenessTheorem}.

With more care we can prove the following explicit result.

\begin{thm}\label{MainDynamicalTheorem}
It holds that 
\begin{equation*}
\begin{split}
\PCF_2\cap\QQ^\tr & = \{-2,-1,0\} \\
\PCF_d\cap\QQ^\tr & =\{-1,0\} \text{ if $d\geq4$ is even}\\
\PCF_d\cap\QQ^\tr & =\{0\} \text{ if $d\geq3$ is odd}.
\end{split}
\end{equation*}  
\end{thm}

Thus the only quadratic PCF maps defined over $\QQ^\tr$ in the unicritical family $x^2+c$ are the squaring map $x^2$, the basilica map $x^2-1$, and the Chebyshev map $x^2-2$.  In the families $x^d+c$ for $d\geq3$, we find only the powering maps $x^d$, as well as $x^d-1$ when $d\geq4$ is even.

Our approach is to turn the capacity-theoretic finiteness result $|\PCF_d\cap\QQ^\tr|<+\infty$ into a quantitative upper bound on the degree of any totally real PCF parameter.  The primary tool is the {\em $n$-diameter}
\begin{equation}\label{nDiamDefinition}
d_n(E):=\sup_{x_1,x_2,\dots,x_n\in E}\,\,\,\prod_{i<j}|x_i-x_j|^{2/n(n-1)}
\end{equation}
associated to a compact subset $E$ of $\CC$ and each integer $n\geq2$.  Thus $d_n(E)$ is the maximal possible value of the geometric mean of the pairwise differences among any choice of $n$ points in $E$.  A collection $x_1,x_2,\dots,x_n$ of $n$ points in $E$ on which the supremum is achieved  are called {\em Fekete points} for $E$; such points are guaranteed to exist by the compactness of $E$. Elementary arguments show that the $n$-diameter has the transformation property $d_n(\alpha E+\beta)=|\alpha|d_n(E)$ for $\alpha,\beta\in \CC$, and the monotonicity property $d_n(E_1)\leq d_n(E_2)$ whenever $E_1\subseteq E_2$.

The sequence $\{d_n(E)\}$ is monotone decreasing, and therefore the limit
\[
d_\infty(E):=\lim_{n\to+\infty}d_n(E)
\] 
exists.  The quantity $d_\infty(E)$ is known as the {\em transfinite diameter} of $E$, and it also coincides with the capacity of $E$; see \cite{MR1334766} $\S$ 5.5. 

We recall the well-known calculation that the transfinite diameter of a real interval $[a,b]$ is a quarter of its length, that is $d_\infty([a,b])=\frac{b-a}{4}$; see \cite{MR1334766} Cor. 5.2.4.  In $\S$~\ref{nDiamSection} we prove the following explicit recursion formula for the $n$-diameter of an interval; this result may be of some independent interest.

\begin{thm}\label{nDiamIntervalThm}
For $n\geq2$, the $n$-diameter of a real interval $[\alpha,\beta]$ is given by
\begin{equation}\label{nDiamInterval}
d_n([\alpha,\beta]) = (\beta-\alpha)D_n^{1/n(n-1)},
\end{equation}
where $\{D_n\}_{n=2}^{\infty}$ is the sequence defined recursively by $D_2=1$ and 
\begin{equation*}
D_n  =  \frac{n^n(n-2)^{n-2}}{2^{2n-2}(2n-3)^{2n-3}}D_{n-1} \hskip1cm (n\geq3).
\end{equation*}
\end{thm}

This result can be easily converted into a closed form expression for $d_n([\alpha,\beta])$, but such a formula would not as useful for our purposes as the recursion formula described in Theorem~\ref{nDiamIntervalThm}.  For later use we give the first few values of $\{D_n\}$ here:
\[
D_2 = 1 \hskip5mm D_3 = \frac{1}{16} \hskip5mm D_4 = \frac{1}{3125} \hskip5mm D_5=\frac{27}{210827008}.
\]

It is immediately evident from the definition (\ref{nDiamDefinition}) that
\begin{equation}\label{nDiamAndDisc}
d_n(E)^{n(n-1)} = \sup_{f}|\disc(f)|
\end{equation}
where the supremum is taken over all monic polynomials $f(X)\in\CC[X]$ of degree $n$ having all $n$ roots in the set $E$.  This fact, together with the ``electrostatic interpretation'' of the zeros of Jacobi polynomials (an idea going back at least to Szeg\H{o} \cite{MR0372517}) is the main idea behind the proof of Theorem~\ref{nDiamIntervalThm}.  The $n$-Fekete points associated to the interval $[-1,1]$ are the endpoints together with the roots of the Jacobi polynomial of degree $n-2$ and weight $(1,1)$.  Like all orthogonal polynomials, the Jacobi polynomials satisfy a recursion as well as a method (due to Schur) to calculate their discriminant; we provide the details in $\S$~\ref{nDiamSection}.

In order to apply Theorem~\ref{nDiamIntervalThm} to the proof of Theorem~\ref{MainDynamicalTheorem}, we consider more generally the following problem:

\begin{quote}
Given a real interval $[\alpha,\beta]$ of length $L=\beta-\alpha<4$, find a bound $n_0$, depending on the length $L$, such that if $\theta$ is an algebraic integer, all of whose embeddings into $\CC$ lie in the real interval $[\alpha,\beta]$, then $[\QQ(\theta):\QQ]< n_0$.
\end{quote}

The requirement that $L<4$ is necessary because if $\zeta$ is a root of unity, then all complex embeddings of $\zeta+\zeta^{-1}$ lie in the real interval $[-2,2]$, thus no such $n_0$ exists for this interval.  On the other hand, it follows from general principles in arithmetic capacity theory that such a bound $n_0$ does exist when $L<4$, but finding a bound written as an explicit expression in $L$ seems to be difficult.  Instead, in Theorem~\ref{TotRealIntervalTheorem} we give a numerical criterion which has practically the same effect when applied in specific examples.  This result allows one to fairly easily calculate such a bound $n_0$ for any interval of some given particular length $L<4$.  

The main idea behind Theorem~\ref{TotRealIntervalTheorem} is the following.  Suppose that $\theta$ is an algebraic integer of degree $n=[\QQ(\theta):\QQ]$, such that all $n$ embeddings of $\theta$ into $\CC$ lie in the real interval $[\alpha,\beta]$.  Then Theorem~\ref{nDiamIntervalThm} provides an explicit upper bound on $|\disc(F_\theta)|$, where $F_\theta(X)\in\ZZ[X]$ is the minimal polynomial of $\theta$ over $\QQ$.  A lower bound on $|\disc(F_\theta)|$ is provided by Minkowski's theorem, and combining these inequalities leads to a contradiction for large enough $n$.  The statement of this result is slightly complicated, so we delay it until $\S$~\ref{AlgIntIntervalSect}.

Using Minkowski's bound is not strictly necessary, as using the trivial lower bound $|\disc(F_\theta)|\geq1$ would also lead to a contradiction for large enough $n$.  But in practice, using Minkowski's bound leads to a smaller value of $n_0$, which can make a significant difference in any intended application.  To illustrate, in our proof of Theorem~\ref{MainDynamicalTheorem} in the $d=2$ case, we are able to use Theorem~\ref{TotRealIntervalTheorem} to show that if $c$ is totally real and $x^2+c$ is PCF, then $[\QQ(c):\QQ]<3$.  We then dispose of the cases $[\QQ(c):\QQ]=1,2$ with elementary arguments.  If instead we had used only the trivial lower bound $|\disc(F_c)|\geq1$, we could only have deduced that $[\QQ(c):\QQ]<6$, leading to a much more difficult computational challenge to finish the proof that $\PCF_2\cap\QQ^\tr=\{-2,-1,0\}$.

Theorem~\ref{TotRealIntervalTheorem} should be of some independent interest and certainly has other uses beyond the proof of Theorem~\ref{MainDynamicalTheorem}.  In a paper under preparation \cite{NoytaptimDynSpace}, the first author considers the problem of classifying precisely which parameters $c\in\QQ$ have the property that the map $x^2+c$ has only finitely many totally real preperiodic points.  A fully definitive solution to this problem is given by the first author \cite{NoytaptimDynSpace} using methods in arithmetic capacity theory, and for several of the parameters $c$, a complete calculation of the set of all totally real preperiodic points for $x^2+c$ can be made using Theorem~\ref{TotRealIntervalTheorem}. 

The plan of this paper is as follows.  In $\S$~\ref{nDiamSection} we prove Theorem~\ref{nDiamIntervalThm}, the recursion formula for the $n$-diameter of an interval. In $\S$~\ref{AlgIntIntervalSect} we use this recursion and Minkowski's bound to prove our main result on totally real algebraic integers in short intervals.  Finally, in $\S$~\ref{ParSpaceSect} we apply Theorem~\ref{nDiamIntervalThm} and Theorem~\ref{TotRealIntervalTheorem} to the proof of Theorem~\ref{MainDynamicalTheorem}, the dynamical application of classifying totally real PCF parameters in unicritical families.  

%%%%%%%%%%%%%%%%%%%
%%%%%%%%%%%%%%%%%%%
%%%%%%%%%%%%%%%%%%%
%%%%%%%%%%%%%%%%%%%

\section{The $n$-diameter of an interval}\label{nDiamSection}

In this section we use properties of Jacobi polynomials to prove Theorem~\ref{nDiamIntervalThm}, the recursion formula for the $n$-diameter of an interval. 

Let $\{P_m(x)\}_{m=0}^{\infty}$ be the family of Jacobi polynomials of weight $(\alpha,\beta)=(1,1)$.  Thus each $P_m(x)$ has real coefficients, $\deg P_m=m$, and the family $\{P_m(x)\}_{m=0}^{\infty}$ is orthogonal with respect to the inner product $\langle f ,g\rangle = \int_{-1}^1f(x)g(x)(1-x^2)dx$.  These assumptions determine each $P_m(x)$ up to a real multiplicative constant, and we choose the normalization in which each polynomial $P_m(x)$ is monic.  Szeg\H{o} \cite{MR0372517} is a standard source for orthogonal polynomials.

We also denote by $\{P_m^*(x)\}_{m=0}^{\infty}$ the family of Jacobi polynomials of weight $(\alpha,\beta)=(1,1)$, but normalized as in \cite{MR0372517} so that $P^*_m(1)=m+1$.  It is shown in \cite{MR0372517} $\S$IV.4.21 that the leading coefficient of $P_m^*(x)$ is given by
\begin{equation*}
L_m=\frac{(2m+2)!}{2^m\,m!(m+2)!}.
\end{equation*}
Thus $P^*_m(x)=L_mP_m(x)$, which implies that
\begin{equation}\label{PmAndLm}
|P_m(\pm1)|=P_m(1)=\frac{m+1}{L_m}=\frac{2^m(m+1)!(m+2)!}{(2m+2)!}.
\end{equation}
It is useful for recursion purposes to calculate the ratio
\begin{equation}\label{PmAtOneRatio}
\frac{P_m(1)}{P_{m-1}(1)}=\frac{m+2}{2m+1},
\end{equation}
which follows from (\ref{PmAndLm}).

By the general theory of orthogonal polynomials, there is an alternate characterization of the family $\{P_m(x)\}_{m=0}^{\infty}$.  For each $m\geq0$, $y=P_m(x)$ is the unique monic polynomial of degree $m$ satisfying the differential equation  
\begin{equation}\label{DiffEq}
(1-x^2)y'' - 4xy' + m(m+3)y=0;
\end{equation}
this is proved in \cite{MR0372517}, Thm. 4.2.2.  

\begin{lem}\label{MonicJacobiRecursion} 
The monic Jacobi polynomials $\{P_m(x)\}_{m=0}^{\infty}$ satisfy the recursion $P_0(x) = 1$, $P_1(x)  = x$, and 
\begin{equation*}
\begin{split}
P_m(x) & = xP_{m-1}(x)-C_mP_{m-2}(x) \hskip1cm (m\geq2)
\end{split}
\end{equation*}
where 
\begin{equation}\label{MonicJacobiRecursionConstants}
C_m=\frac{m^2-1}{4m^2-1}.
\end{equation}
Moreover, the discriminant of Jacobi polynomials satisfy the recursion $|\disc(P_1)|  =1$ and 
\begin{equation}\label{MonicJacobiDiscrimiantRecursion}
|\disc(P_m)| =\frac{m^m(m+2)^{m-2}}{(2m+1)^{2m-3}}|\disc(P_{m-1})| \hskip1cm (m\geq2).
\end{equation}
\end{lem}

\begin{proof}
The calculations of $P_0(x)$ and $P_1(x)$ follow from the monic assumption together with the orthogonality $\langle P_0,P_1\rangle=0$.  All orthogonal polynomials satisfy a recursion of the type
\begin{equation*}
P_m(x) = (A_mx+B_m)P_{m-1}(x)-C_mP_{m-2}(x) \hskip1cm (m\geq2),
\end{equation*}
see \cite{MR0372517} $\S$ 3.1.  That $A_m=1$ follows from the monic assumption.  That $B_m=0$ follows from the fact that $P_m(x)$ is either even or odd according to the parity of $m$ (see \cite{MR0372517} $\S$ 4.1), and therefore each $P_m(x)$ has vanishing $x^{m-1}$ term. To calculate the $C_m$, we evaluate the recursion at $x=1$, we find $P_m(1) = P_{m-1}(1)-C_mP_{m-2}(1)$ and thus 
\begin{equation*}
C_m=\frac{P_{m-1}(1)-P_{m}(1)}{P_{m-2}(1)} = \frac{m^2-1}{4m^2-1}.
\end{equation*}
by (\ref{PmAndLm}).

To prove the recursion (\ref{MonicJacobiDiscrimiantRecursion}) for the discriminant, we use a method of Schur (see \cite{MR0372517} $\S$ 6.71).  For each $m\geq2$, set 
\[
\Delta_m=|\Res(P_m,P_{m-1})|=\prod_{P_m(\alpha)=0}|P_{m-1}(\alpha)|=\prod_{P_{m-1}(\beta)=0}|P_{m}(\beta)|.
\]
The recursion $P_m(x) = xP_{m-1}(x)-C_mP_{m-2}(x)$ shows that 
\begin{equation}\label{DeltaRecursion}
\Delta_m=C_m^{m-1}\Delta_{m-1}.
\end{equation}
To apply this recursion to the discriminant of the Jacobi polynomials $P_m(x)$, we note also that
\[
(1-\alpha^2)P_m'(\alpha)=N_mP_{m-1}(\alpha) \hskip1cm \text{whenever }P_{m}(\alpha)=0,
\]
where 
\[
N_m=\frac{mP_m(1)}{P_{m-1}(1)}=\frac{m(m+2)}{2m+1};
\]
this is derived in \cite{MR0372517} $\S$ 4.5.  Therefore
\begin{equation*}
\begin{split}
|\disc(P_m)| & = |\Res(P_m,P_m')| \\
	& = \prod_{P_m(\alpha)=0}|P_m'(\alpha)| \\
	& = N_m^m\prod_{P_m(\alpha)=0}|1-\alpha^2|^{-1}|P_{m-1}(\alpha)| \\
	& = N_m^m|P_m(1)|^{-1}|P_m(-1)|^{-1}\prod_{P_m(\alpha)=0}|P_{m-1}(\alpha)| \\
	& = N_m^mP_m(1)^{-2}\Delta_m \\
\end{split}
\end{equation*}
We conclude that 
\begin{equation*}
\begin{split}
\frac{|\disc(P_m)|}{|\disc(P_{m-1})|} & = \frac{N_m^m}{N_{m-1}^{m-1}}\cdot\frac{P_{m-1}(1)^2}{P_m(1)^2}\cdot\frac{\Delta_m}{\Delta_{m-1}} \\
	& = \frac{m^m(m+2)^{m-2}}{(2m+1)^{2m-3}},
\end{split}
\end{equation*}
by substituting the definition of $N_m$ together with (\ref{PmAtOneRatio}) and (\ref{DeltaRecursion}) and simplifying; this completes the proof of the recursion (\ref{MonicJacobiDiscrimiantRecursion}).
\end{proof}

The following lemma may be viewed as a consequence of the ``electrostatic interpretation'' of the zeros of Jacobi polynomials, an idea which goes back at least to Szeg\H{o} \cite{MR0372517}.  For completeness we sketch the proof given in \cite{MR0372517} $\S$ VI.6.7.
	
\begin{lem}\label{feketeoninterval} 
For each $n\geq2$, the supremum
\[
\sup_{x_1,x_2,\dots,x_n\in [-1,1]}\prod_{i<j}|x_i-x_j|^{2}
\]
is achieved when the points $x_1,x_2,\dots,x_n$ are the roots of the polynomial 
\[
Q_n(x)=(x^2-1)P_{n-2}(x).
\]
In particular $d_n([-1,1])^{n(n-1)}=|\disc(Q_n)|$.
\end{lem}
	
\begin{proof} 
Let $-1\leq x_1<x_2<\dots<x_{n-1}<x_n\leq 1$ be an ordered list of points at which the desired supremum is achieved; such points exist by compactness.  Since the expression inside the supremum increases as $x_n$ increases, we must have $x_n=1$, and similarly $x_1=-1$.

For each index $2\leq k\leq n-1$, by the maximality assumption, the point $x_k$ must be a critical point for the function
\[
g(x)=\log|x+1|+\log|x-1|+\sum_{\substack{2\leq j\leq n-1 \\ j\neq k}}\log|x-x_j|.
\] 
and therefore
\begin{equation}\label{minimizeenergyeq}
\frac{1}{x_k+1}+\frac{1}{x_k-1}+\sum_{\substack{2\leq j\leq n-1 \\ j\neq k}}\frac{1}{x_k-x_j}=0.
\end{equation}

Define $f(x)=\prod_{j=2}^{n-1}(x-x_j)$, so the proof will be complete if we show that $f(x)=P_{n-2}(x)$.  For each $2\leq k\leq n-1$, let $f_k(x)=f(x)/(x-x_k)$.  Thus
\[
\frac{f_k'(x)}{f_k(x)}=\sum_{\substack{2\leq j\leq n-1 \\ j\neq k}}\frac{1}{x-x_j}
\] 
and
\begin{equation*}
\begin{split}
f'(x) & =(x-x_k)f_k'(x)+f_k(x) \\
f''(x) & =(x-x_k)f_k''(x)+2f_k'(x).
\end{split}
\end{equation*}
For each $2\leq k\leq n-1$ we obtain $f'(x_k)=f_k(x_k)$ and $f''(x_k)=2f_k'(x_k)$, and therefore  
\[
\frac{f''(x_k)}{2f'(x_k)}=\frac{f_k'(x_k)}{f_k(x_k)}=\sum_{\substack{2\leq j\leq n-1 \\ j\neq k}}\frac{1}{x_k-x_j}.
\]
Combining this with (\ref{minimizeenergyeq}) we obtain 
\begin{align*}
\frac{f''(x_k)}{2f'(x_k)}+\frac{1}{x_k-1}+\frac{1}{x_k+1}=0
\end{align*}
which simplifies to $(x^2_k-1)f''(x_k)+4x_kf'(x_k)=0$.  This tells us that the polynomials $f(x)$ and $(x^2-1)f''(x)+4xf'(x)$ share the same roots.  Since the former is monic and the latter has leading coefficient $(n-2)(n-3)+4(n-2)=(n+1)(n-2)$, we obtain 
\[
(x^2-1)f''(x)+4xf'(x)=(n+1)(n-2)f(x).
\]
Since $f(x)$ is monic and satisfies the differential equation (\ref{DiffEq}) with $m=n-2$, we conclude that $f(x)=P_{n-2}(x)$, completing the proof.
\end{proof}

\begin{lem}\label{MonicModifiedJacobiRecursion} 
The polynomials $\{Q_n(x)\}_{n=2}^{\infty}$ defined by $Q_n(x)=(x^2-1)P_{n-2}(x)$ satisfy the recursion $|\disc(Q_2)|= 4$ and 
\begin{equation*}
|\disc(Q_n)| = \frac{n^n(n-2)^{n-2}}{(2n-3)^{2n-3}} |\disc(Q_{n-1})| \hskip1cm (n\geq3).
\end{equation*}
\end{lem}

\begin{proof} 
We have $|\disc(Q_2)|=|\disc(x^2-1)|=4$.  If $f(x),g(x)\in\CC[x]$ are monic polynomials, recall the well-known identity 
\[
|\disc(fg)|=|\disc(f)||\Res(f,g)|^2|\disc(g)|
\]
for the discriminant of their product, where $|\Res(f,g)|=\prod_{f(\alpha)=0}|g(\alpha)|$.  Using this and the fact that $|P_m(-x)|=|P_m(x)|$, we have
\begin{equation*}
\begin{split}
|\disc(Q_n)| & = |\disc(x^2-1)||\Res(x^2-1,P_{n-2})|^2|\disc(P_{n-2})| \\
	& = 4|P_{n-2}(1)|^4|\disc(P_{n-2})|,
\end{split}
\end{equation*}
and therefore when $n\geq3$ we use (\ref{PmAtOneRatio}) and Lemma~\ref{MonicJacobiRecursion} to obtain
\begin{equation*}
\begin{split}
\frac{|\disc(Q_n)|}{|\disc(Q_{n-1})|} & = \frac{|P_{n-2}(1)|^4}{|P_{n-3}(1)|^4}\frac{|\disc(P_{n-2})|}{|\disc(P_{n-3})|} \\
	& = \frac{n^n(n-2)^{n-2}}{(2n-3)^{2n-3}}, \\
\end{split}
\end{equation*}
which is the desired recursion.
\end{proof}

\begin{proof}[Proof of Theorem~\ref{nDiamIntervalThm}]
For $n\geq2$ we define $D_n=2^{-n(n-1)}d_n([-1,1])^{n(n-1)}$.  Thus 
\begin{equation*}
d_n([-1,1]) = 2D_n^{1/n(n-1)},
\end{equation*}
and more generally, in view of the transformation property $d_n(r S+t)=|r|d_n(S)$ for $r,t\in \RR$, for $L=(\beta-\alpha)/2$ we have
\begin{equation*}
d_n([\alpha,\beta]) = d_n([-L,L])=Ld_n([-1,1])=L2D_n^{1/n(n-1)}=(\beta-\alpha)D_n^{1/n(n-1)}.
\end{equation*}

We now just have to show that the sequence $\{D_n\}_{n=2}^\infty$ satisfies the initial condition and recurrence relation given in the statement of the theorem.  Clearly $D_2=2^{-2}d_2([-1,1])^2=1$.  Using Lemma~\ref{feketeoninterval} and Lemma~\ref{MonicModifiedJacobiRecursion}, for $n\geq3$ we have
\begin{equation*}
\begin{split}
\frac{D_n}{D_{n-1}} & = \frac{2^{-n(n-1)}d_n([-1,1])^{n(n-1)}}{2^{-(n-1)(n-2)}d_{n-1}([-1,1])^{(n-1)(n-2)}} \\
	& = 2^{-2n+2}\frac{|\disc(Q_n)|}{|\disc(Q_{n-1})|} \\
	& = \frac{n^n(n-2)^{n-2}}{2^{2n-2}(2n-3)^{2n-3}},
\end{split}
\end{equation*}
which is the stated recurrence relation for the sequence $\{D_n\}_{n=2}^\infty$.
\end{proof}

%%%%%%%%%%%%%%%%%%%
%%%%%%%%%%%%%%%%%%%
%%%%%%%%%%%%%%%%%%%
%%%%%%%%%%%%%%%%%%%

\section{Algebraic integers with conjugates in a short interval}\label{AlgIntIntervalSect}

In this section we prove our main result on totally real algebraic integers with all conjugates in a short interval.  We first need two preliminary lemmas.

\begin{lem}\label{RecursiveSequencesLemma}
Let $n_0$ be an integer and let $\{a_n\}_{n=n_0}^{\infty}$ and $\{b_n\}_{n=n_0}^{\infty}$ be sequences of positive real numbers.  If $a_{n_0}<b_{n_0}$, and  $a_{n+1}/a_{n}<b_{n+1}/b_{n}$ for all $n\geq n_0$, then $a_n<b_n$ for all $n\geq n_0$.
\end{lem}
\begin{proof}
The proof is an induction with base case $n=n_0$.  If $a_{n}<b_{n}$ for some $n\geq n_0$, then
\[
\frac{a_{n+1}}{b_{n+1}}=\frac{a_{n+1}/a_n}{b_{n+1}/b_{n}}\cdot\frac{a_n}{b_n}<1
\]
and therefore $a_{n+1}<b_{n+1}$.
\end{proof}

\begin{lem}\label{MinkowskiLemma}
Let $F(X)\in\ZZ[X]$ be a monic, irreducible polynomial of degree $n\geq2$ with $n$ real roots.  Then $|\disc(F)|\geq n^{2n}/n!^2$. 
\end{lem}
\begin{proof}
Let $\theta\in\RR$ be a root of $F(X)$, let $K=\QQ(\theta)$, and let $\Delta_K\in\ZZ$ be the discriminant of $K$.  Let $\sigma_i:K\hookrightarrow\RR$ $(i=1,2,\dots,n)$ denote the $n$ distinct embeddings of $K$ into $\RR$.  By a standard calculation, $|\disc(F)|=|\det V|^2$, where $V=(\sigma_i(\theta^j))$, and thus
\[
|\disc(F)|=|\det V|^2=|\Ocal_K/\ZZ[\theta]|^2\cdot |\Delta_K|\geq |\Delta_K|\geq n^{2n}/n!^2.
\]  
Here the final inequality is Minkoswki's lower bound on the discriminant in the special case of a totally real number field $K$.
\end{proof}

We are now ready to state and prove the main result of this section.  

\begin{thm}\label{TotRealIntervalTheorem}
Let $[\alpha,\beta]$ be a real interval of length $0<\beta-\alpha<4$.  Define sequences $\{a_n\}_{n=2}^\infty$ and $\{b_n\}_{n=2}^\infty$ by
\begin{equation*}
\begin{split}
a_n & = d_n([\alpha,\beta])^{n(n-1)} =(\beta-\alpha)^{n(n-1)} D_n \\
b_n & = \frac{n^{2n}}{n!^2},
\end{split}
\end{equation*}
where the sequence $\{D_n\}$ is defined in Theorem~\ref{nDiamIntervalThm}.  Suppose that there exists an integer $n_0\geq2$ with the properties that
\begin{equation}
a_{n_0}<b_{n_0} \hskip1cm \text{ and } \hskip1cm \frac{a_{n_0+1}}{a_{n_0}}<\frac{b_{n_0+1}}{b_{n_0}}.
\end{equation}
If $\theta$ is an algebraic integer such that the minimal polynomial $F_\theta(X)\in\ZZ[X]$ of $\theta$ over $\QQ$ has all $n$ roots in the interval $[\alpha,\beta]$, then $[\QQ(\theta):\QQ]<n_0$.
\end{thm}

\begin{proof}
We are going to show that $a_n<b_n$ for all $n\geq n_0$.  This will be sufficient to prove the theorem, because if this holds, and if $\theta$ is an algebraic integer of degree $n=[\QQ(\theta):\QQ]$, with minimal polynomial $F_\theta(X)\in\ZZ[X]$ having all $n$ roots in the interval $[\alpha,\beta]$, then Lemma~\ref{MinkowskiLemma} and the definition of the $n$-diameter of the interval $[\alpha,\beta]$ respectively imply the lower and upper bounds
\[
b_n\leq |\disc(F_\theta)| \leq a_n.
\]
But as $a_n<b_n$ for all $n\geq n_0$, it follows that $n<n_0$ which is the desired result of the theorem.

To show that $a_n<b_n$ for all $n\geq n_0$, note that since it holds for $n=n_0$ by assumption, according to Lemma~\ref{RecursiveSequencesLemma} we just have to show that that $\frac{a_{n+1}}{a_{n}}<\frac{b_{n+1}}{b_{n}}$ for all $n\geq n_0$, or equivalently that
\begin{equation}\label{GrowthRateIneq}
\frac{a_n}{a_{n-1}}<\frac{b_n}{b_{n-1}}
\end{equation}
for all $n\geq n_0+1$.  We are also given that $(\ref{GrowthRateIneq})$ holds for $n=n_0+1$ by assumption.  So we just have to show that $(\ref{GrowthRateIneq})$ holds more generally for all $n\geq n_0+1$.  

By an elementary calculation,
\begin{equation*}
\begin{split}
\frac{b_n}{b_{n-1}}=\frac{n^{2n}}{n!^2}\cdot\frac{(n-1)!^2}{(n-1)^{2(n-1)}}=\frac{n^{2n-2}}{(n-1)^{2n-2}}. 
\end{split}
\end{equation*}
Abbreviating $L=\beta-\alpha$ and using Theorem~\ref{nDiamIntervalThm} we have
\begin{equation*}
\begin{split}
\frac{a_n}{a_{n-1}} & = \frac{L^{n(n-1)}D_n}{L^{(n-1)(n-2)}D_{n-1}} = \frac{L^{2n-2}n^n(n-2)^{n-2}}{2^{2n-2}(2n-3)^{2n-3}}.
\end{split}
\end{equation*}
Thus (\ref{GrowthRateIneq}) is equivalent to 
\begin{equation}\label{GrowthRateIneq2}
\frac{L^{2n-2}n^n(n-2)^{n-2}}{2^{2n-2}(2n-3)^{2n-3}}<\frac{n^{2n-2}}{(n-1)^{2n-2}}
\end{equation}
and an elementary manipulation shows that this is in turn equivalent to the inequality 
\begin{equation}\label{GrowthRateIneq3}
2\log(L/2) < \frac{1}{n-1}\log\left(\frac{n^{n-2}(2n-3)^{2n-3}}{(n-1)^{2n-2}(n-2)^{n-2}}\right).
\end{equation}

By assumption (\ref{GrowthRateIneq3}) holds for $n=n_0+1$, and in order to conclude that it holds for all $n\geq n_0+1$ we just have to show that the function defined by 
\begin{equation*}
h(x) =\frac{1}{x-1}\log\left(\frac{x^{x-2}(2x-3)^{2x-3}}{(x-1)^{2x-2}(x-2)^{x-2}}\right)
\end{equation*}
is increasing for $x\geq3$.  An elementary calculation shows that
\begin{equation*}
h'(x) =\frac{1}{(x-1)^2}\left(\frac{2}{x}+\log(2x-3)+\log x-\log(x-2)-2\right).
\end{equation*}
A simple calculus argument shows that $h'(x)>0$ for all $x\geq3$.  Thus $h(x)$ is increasing and we conclude that (\ref{GrowthRateIneq3}) and hence (\ref{GrowthRateIneq}) holds for all $n\geq n_0+1$, completing the proof of the theorem.
\end{proof}

The following sample application of Theorem~\ref{TotRealIntervalTheorem} shows that it can be combined with elementary arguments to give sharp results.  This corollary will also immediately imply the degree $d\geq3$ case of Theorem~\ref{MainDynamicalTheorem}.  Given a real interval $[\alpha,\beta]$, denote by $T_{[\alpha,\beta]}$ the set of all algebraic integers $\theta$ with the property that all $[\QQ(\theta):\QQ]$ embeddings of $\theta$ into $\CC$ lie in the real interval $[\alpha,\beta]$.

\begin{cor}\label{TotRealIntervalCor}
If $[\alpha,\beta]$ is a real interval of length $\beta-\alpha<\sqrt5$, then $T_{[\alpha,\beta]}\subseteq\ZZ$.
\end{cor}

This corollary is sharp in the sense that a counterexample exists if $\beta-\alpha=\sqrt{5}$.  The irrational algebraic integer $\theta=\frac{1+\sqrt{5}}{2}$ and its algebraic conjugate $\theta'=\frac{1-\sqrt{5}}{2}$ both lie in the interval $[\theta',\theta]$, which has length $\sqrt{5}$.

\begin{proof}[Proof of Corollary~\ref{TotRealIntervalCor}]
We wish to prove that if $\theta\in T_{[\alpha,\beta]}$, then $\theta\in\QQ$.  That in fact $\theta\in\ZZ$ follows at once since $\theta$ is assumed to be an algebraic integer.  The strategy of proof is to use Theorem~\ref{nDiamIntervalThm} and Theorem~\ref{TotRealIntervalTheorem} to prove that $[\QQ(\theta):\QQ]<3$, and then to use elementary arguments to show that $[\QQ(\theta):\QQ]\neq2$.

We apply Theorem~\ref{TotRealIntervalTheorem}, thus for each $n\geq2$ we set
\begin{equation*}
\begin{split}
a_n & = d_n([\alpha,\beta])^{n(n-1)}=(\beta-\alpha)^{n(n-1)}D_n \\
b_n & = \frac{n^{2n}}{n!^2}
\end{split}
\end{equation*}
using Theorem~\ref{nDiamIntervalThm} to calculate the $n$-diameters of the interval.  In particular
\begin{equation*}
\begin{split}
a_2 & = d_2([\alpha,\beta])^{2(2-1)}=(\beta-\alpha)^{2}D_2=(\beta-\alpha)^{2} \\
a_3 & = d_3([\alpha,\beta])^{3(3-1)}=(\beta-\alpha)^{6}D_3=(\beta-\alpha)^{6}\frac{1}{16} \\
a_4 & = d_4([\alpha,\beta])^{4(4-1)}=(\beta-\alpha)^{12}D_4=(\beta-\alpha)^{12}\frac{1}{3125}
\end{split}
\end{equation*}
and 
\begin{equation*}
\begin{split}
b_2 & = \frac{2^{4}}{2!^2}=4 \\
b_3 & = \frac{3^{6}}{3!^2}=\frac{81}{4} \\
b_4 & = \frac{4^{8}}{4!^2}=\frac{1024}{9} 
\end{split}
\end{equation*}
using $D_2=1$, $D_3=\frac{1}{16}$, and $D_4=\frac{1}{3125}$ as described after the statement of Theorem~\ref{nDiamIntervalThm}.  

We have $a_3<b_3$, as this inequality simplifies to $(\beta-\alpha)^3<18$, which holds because $(\beta-\alpha)^3<(\sqrt{5})^3\approx11.18...$.  We also have $\frac{a_4}{a_3}<\frac{b_4}{b_3}$ because this inequality simplifies to 
\[
(\beta-\alpha)^6<\frac{800000}{729}\approx1097.393...,
\]
which holds because $(\beta-\alpha)^6<\sqrt{5}^6=125$.  We conclude using Theorem~\ref{TotRealIntervalTheorem} that if $\theta\in T_{[\alpha,\beta]}$, then $[\QQ(\theta):\QQ]<3$.

It remains only to prove that $[\QQ(\theta):\QQ]\neq2$ for $\theta\in T_{[\alpha,\beta]}$.  If in fact $[\QQ(\theta):\QQ]=2$ and $\theta$ has minimal polynomial $F(X)=X^2+aX+b\in\ZZ[X]$, then $\disc(F)=a^2-4b$ must be a positive nonsquare, as $F(X)$ is irreducible over $\QQ$ and has two real roots.  But $\disc(F)=a^2-4b=2$ and $\disc(F)=a^2-4b=3$ are both impossible as $2$ and $3$ are not squares modulo $4$.  So $\disc(F)\geq5$.  If $\theta'$ denotes the algebraic conjugate of $\theta$, then $\disc(F)=(\theta-\theta')^2$ and hence $|\theta-\theta'|\geq\sqrt{5}$.  But this contradicts the assumption that both $\theta$ and $\theta'$ are elements of the interval $[\alpha,\beta]$ which has length $<\sqrt{5}$.

We have proved that $T_{[\alpha,\beta]}\subseteq\QQ$ and hence $T_{[\alpha,\beta]}\subseteq\ZZ$ as $T_{[\alpha,\beta]}$ contains only algebraic integers.
\end{proof}

%%%%%%%%%%%%%%%%%%%
%%%%%%%%%%%%%%%%%%%
%%%%%%%%%%%%%%%%%%%
%%%%%%%%%%%%%%%%%%%

\section{Application to post-critically finite polynomials in unicritical families}\label{ParSpaceSect}

Our goal in this section is to apply Theorem~\ref{nDiamIntervalThm} and Theorem~\ref{TotRealIntervalTheorem} to the proof of Theorem~\ref{MainDynamicalTheorem}, the dynamical application of classifying totally real PCF parameters in unicritical families.  

Although superseded by Theorem~\ref{MainDynamicalTheorem}, we first include the following qualitative finiteness result, because it has a fairly elementary proof and illustrates the main ideas behind the proof of Theorem~\ref{MainDynamicalTheorem} in a simple way.  Recall that $\QQ^\tr$ denotes the maximal totally real subfield of $\Qbar$, and that for each integer $d\geq2$ we denote
\[
\PCF_d=\{c\in\Qbar\mid f(x)=x^d+c\text{ is PCF}\}.
\]

\begin{thm}\label{FinitenessTheorem}
For each $d\geq2$, the set $\PCF_d\cap\QQ^\tr$ is finite.
\end{thm}
\begin{proof} 
We begin with the $d=2$ case.  For a parameter $c$, define $f_c(x)=x^2+c$, and note that if $f_c$ is PCF then $0$ is preperiodic and hence $f_c^i(0)=f_c^j(0)$ for some integers $i<j$.  In other words $c$ is a root of the polynomial $G(X)=f_X^j(0)-f_X^i(0)$, which is monic with integral coefficients, and so $c$ must be an algebraic integer.   Let $F_{c}(X)\in\ZZ[X]$ be the monic minimal polynomial of $c$.  Since the parameter $c$ is a root of the monic integral polynomial $f^j_{X}(0)-f^i_{X}(0)$ in $X$ for some $i<j$, it follows that $F_{c}(X)$ is a divisor of $f^j_{X}(0)-f^i_{X}(0)$, and hence all of the roots of $F_{c}(X)$ have the property that $f_c$ is PCF.  It follows from this and the totally real hypothesis that all of the roots of $F_{c}(X)$ are real and also in the Mandelbrot set 
\[
\Mcal_2=\{c\in\CC\mid x^2+c\text{ has bounded critical orbit}\}.
\]
It is well known that $\RR\cap\Mcal_2=[-2,1/4]$.  Since the transfinite diameter of an interval is one quarter of its length (\cite{MR1334766} Cor. 5.2.4), we have $d_\infty(\RR\cap\Mcal_2)<1$ and hence $\RR\cap\Mcal_2$ contains only finitely many complete sets of conjugates of algebraic integers, completing the proof that $\PCF_d\cap\QQ^\tr$ is finite.

Because it gives a simple illustration of the ideas used in the proof of Theorem~\ref{MainDynamicalTheorem}, we can quickly describe the proof of Fekete's theorem in this special case.  Suppose on the contrary that $\{c_m\}$ is an infinite sequence of distinct points in $\QQ^{\mathrm{tr}}$ such that $f_{c_m}$ is PCF for all $m\geq1$.  Since the $c_m$ are algebraic integers and all of the complex embeddings of each $c_m$ are in $[-2,1/4]$, the sequence $\{c_m\}$ has bounded height and so by Northcott's property, we know that $n_m:=[\QQ(c_m): \QQ]\rightarrow  \infty$ as $m\rightarrow +\infty$.  Since $c_m$ is an algebraic integer, $\disc(F_m)$ is a nonzero rational integer.  Therefore
\begin{align*}
1 \leq |\text{disc}(F_m)|^{1/n_m(n_m-1)} \leq d_{n_m}([-2,1/4])\to d_\infty([-2,1/4])=9/16,
\end{align*} 
a contradiction as $m\to+\infty$.  

For $d\geq3$ the same proof works, using properties of the degree $d$ analogue 
\[
\Mcal_d=\{c\in\CC\mid x^d+c\text{ has bounded critical orbit}\}
\]
of the Mandelbrot set.  It follows from work of Paris\'e-Rochon \cite{MR3654403} and Paris\'e-Ransford-Rochon \cite{PariseRansfordRochon} that for all $d\geq3$, the set $\Mcal_d\cap\RR$ is an interval of length less than $4$.  Thus $d_\infty(\Mcal_d\cap\RR)<1$ for $d\geq3$ and the same contradiction is obtained.
\end{proof}

\begin{proof}[Proof of Theorem~\ref{MainDynamicalTheorem} in the case $d=2$] We seek to prove that
\begin{equation}\label{MainDynamicalCalculation1}
\begin{split}
\PCF_2\cap\QQ^\tr & = \{-2,-1,0\};
\end{split}
\end{equation}  
that is, that the only totally real $c$ for which $x^2+c$ is PCF are $c=-2,-1,0$.

Suppose that $c\in\Qbar$ is totally real and that $x^2+c$ is PCF.  In particular, $c$ must be an algebraic integer (as explained in the proof of Theorem~\ref{FinitenessTheorem}).  Since all of the $\Gal(\Qbar/\QQ)$-conjugates of $c$ in $\Qbar$ are also PCF and are real when embedded into $\CC$, the minimal polynomial of $c$ over $\QQ$ has all roots in $\Mcal_2\cap\RR$, where $\Mcal_2$ is the ordinary Mandelbrot set.  Recall that $\Mcal_2\cap\RR=[-2,1/4]$.

We now proceed in an argument similar the proof of Corollary~\ref{TotRealIntervalCor}, following these steps:

\begin{itemize}
\item[]Step 1: We use Theorem~\ref{nDiamIntervalThm} and Theorem~\ref{TotRealIntervalTheorem} to prove that $[\QQ(c):\QQ]<3$.
\item[]Step 2: We use elementary arguments to show that $[\QQ(c):\QQ]\neq2$.
\item[]Step 3: We recall a well-known argument to conclude that $c\in\{-2,-1,0\}$.
\end{itemize}

\underline{Step 1:} We apply Theorem~\ref{TotRealIntervalTheorem} with $L=9/4$, the length of the interval $[-2,1/4]$.  For each $n\geq2$ set
\begin{equation*}
\begin{split}
a_n & = d_n([-2,1/4])^{n(n-1)}=(9/4)^{n(n-1)}D_n \\
b_n & = \frac{n^{2n}}{n!^2},
\end{split}
\end{equation*}
where the sequence $\{D_n\}$ is defined in Theorem~\ref{nDiamIntervalThm}.  We calculate
\begin{equation*}
\begin{split}
a_2 & = d_2([-2,1/4])^{2(2-1)}=(9/4)^{2}D_2=\frac{81}{16}=5.0625 \\
a_3 & = d_3([-2,1/4])^{3(3-1)}=(9/4)^{6}D_3=\frac{531441}{65536}\approx 8.109... \\
a_4 & = d_4([-2,1/4])^{4(4-1)}=(9/4)^{12}D_4=\frac{282429536481}{52428800000}\approx 5.386... \\
\end{split}
\end{equation*}
and 
\begin{equation*}
\begin{split}
b_2 & = \frac{2^{4}}{2!^2}=4 \\
b_3 & = \frac{3^{6}}{3!^2}=\frac{81}{4}=20.25 \\
b_4 & = \frac{4^{8}}{4!^2}=\frac{1024}{9}\approx113.777...
\end{split}
\end{equation*}
using $D_2=1$, $D_3=\frac{1}{16}$, and $D_4=\frac{1}{3125}$ as described after the statement of Theorem~\ref{nDiamIntervalThm}.  Moreover, using these calculations we have 
\begin{equation*}
\begin{split}
\frac{a_4}{a_3} & =\frac{282429536481/52428800000}{531441/65536}=\frac{531441}{800000}\approx0.664... \\
\frac{b_4}{b_3} & =\frac{1024/9}{81/4}=\frac{4096}{729}\approx5.618...
\end{split}
\end{equation*}
and thus $\frac{a_4}{a_3}<\frac{b_4}{b_3}$.  We conclude using Theorem~\ref{TotRealIntervalTheorem} with $n_0=3$ that $[\QQ(c):\QQ]<3$.  

\underline{Step 2:}  Assume that $[\QQ(c):\QQ]=2$.  Let $F(X)=X^2+aX+b\in\ZZ[X]$ be the minimal polynomial of $c$ over $\QQ$, thus $a,b\in\ZZ$ since $c$ is an algebraic integer. Since $c$ is totally real and $x^2+c$ is PCF, we know that both complex roots $c_1,c_2$ of $F(X)$ are in the real interval $\Mcal_2\cap\RR=[-2,1/4]$.  Thus 
\[
0<\disc(F)=a^2-4b=(c_1-c_2)^2\leq(9/4)^2=5.0625.
\]
Moreover $\disc(F)$ is not a square since $F(X)$ is irreducible, so $\disc(F)$ is either $2$, $3$, or $5$.  We cannot have $\disc(F)=2$ or $\disc(F)=3$ because $\disc(F)=a^2-4b$ and neither $2$ nor $3$ is a square modulo $4$.  

So we must have $\disc(F)=a^2-4b=5$; in particular $a$ must be odd.  Since both $c_1$ and $c_2$ are in the interval $[-2,1/4]$, we have $a=-(c_1+c_2)\in[-1/2,4]$, so either $a=1$ or $a=3$.  If $a=1$ then $b=-1$ and $F(X)=X^2+X-1$; but one of the roots of this polynomial, $c_1=\frac{-1+\sqrt{5}}{2}\approx0.618...$ is not contained in $[-2,1/4]$, which gives a contradiction.  If $a=3$ then $b=1$ and $F(X)=X^2+3X+1$; but one of the roots of this polynomial, $c_1=\frac{-3-\sqrt{5}}{2}\approx-2.618...$ is not contained in $[-2,1/4]$, which again gives a contradiction, completing the proof that $[\QQ(c):\QQ]\neq2$.

\underline{Step 3:} We now know that $[\QQ(c):\QQ]=1$ and hence $c\in \QQ$.  Since $c$ is an algebraic integer and hence a rational integer, and $c\in\Mcal_2\cap\RR=[-2,1/4]$, we conclude that $c\in\{-2,-1,0\}$.  It is elementary to check that all three of $x^2-2$, and $x^2-1$, and $x^2$ are PCF, concluding the proof of (\ref{MainDynamicalCalculation1}).
\end{proof}

To prove Theorem~\ref{MainDynamicalTheorem} in the case $d\geq3$, we need to understand the intersection $\Mcal_d\cap\RR$ of the degree $d$ generalized Mandelbrot set $\Mcal_d$ with the real line.  It has been shown by Paris\'e-Rochon \cite{MR3654403} when $d\geq3$ is odd, and by Paris\'e-Ransford-Rochon \cite{PariseRansfordRochon} when $d\geq4$ is even, that
\begin{equation}\label{HigherDegreeIntervalDef}
\Mcal_d\cap\RR = 
\begin{cases}
[-a_d,a_d] & \text{ when $d\geq3$ is odd} \\
[-b_d,a_d] & \text{ when $d\geq4$ is even},
\end{cases}
\end{equation}  
where $a_d=(d-1)/(d^{d/(d-1)})$ and $b_d=2^{1/(d-1)}$. 

\begin{lem}\label{MultiLengthLem}
For each $d\geq3$, the interval $\Mcal_d\cap\RR$ has length less than $\sqrt5$.
\end{lem}
\begin{proof}
First, we have
\[
a_d=\frac{d-1}{d^{d/(d-1)}}<1
\]
for all $d\geq2$.  Indeed, this is algebraically equivalent to the inequality 
\[
(d-1)\log(d-1)<d\log d,
\] 
which follows from the fact that $x\mapsto x\log x$ is increasing for $x\geq1$.  We conclude that when $d\geq3$ is odd, the interval $\Mcal_d\cap\RR=[-a_d,a_d]$ has length $<2<\sqrt{5}$.

Now consider the case  $d\geq4$ even. We want to show that $a_d+b_d<\sqrt{5}$.  We can just check numerically that $a_4+b_4=3/(4^{4/3}) + 2^{1/3}\approx 1.732...<\sqrt{5}$, while for $d\geq6$ (even) we have $d-1\geq5$ and so 
\[
a_d+b_d=a_d + 2^{1/(d-1)} \leq a_d + 2^{1/5} < 1+2^{1/5}\approx 2.148<\sqrt{5}.
\]
\end{proof}

\begin{proof}[Proof of Theorem~\ref{MainDynamicalTheorem} in the case $d\geq3$] We seek to prove that
\begin{equation}\label{MainDynamicalCalculation2}
\begin{split}
\PCF_d\cap\QQ^\tr & =\{-1,0\} \text{ if $d\geq4$ is even} \\
\PCF_d\cap\QQ^\tr & =\{0\} \text{ if $d\geq3$ is odd};
\end{split}
\end{equation}
that is, for $d\geq3$, the only PCF unicritical maps $x^d+c$ for $c\in\QQ^\tr$ are $x^d-1$ and $x^d$ when $d\geq4$ is even, and $x^d$ when $d\geq3$ is odd.  

Suppose that $c\in \PCF_d\cap\QQ^\tr$.  By the same argument described in the degree $d=2$ case, $c$ must be an algebraic integer, and all of the algebraic conjugates of $c$ in $\CC$ lie in the real interval $\Mcal_d\cap\RR$.  Since these intervals have length $<\sqrt5$ for all $d\geq3$ by Lemma~\ref{MultiLengthLem}, it follows from Corollary~\ref{TotRealIntervalCor} that $c\in\ZZ$.  Recalling that $a_d<1$ and $b_d=2^{1/(d-1)}$, when $d\geq3$ is odd the only integer in $\Mcal_d\cap\RR=[-a_d,a_d]$ is $c=0$, and when $d\geq4$ is even the only integers in $\Mcal_d\cap\RR=[-b_d,a_d]$ are $c=-1,0$, completing the proof.
\end{proof}

%\bibliographystyle{siam}

%\bibliography{CJP}

\end{document}